\documentclass[10pt,electronic]{amsart}        

\usepackage[swedish, english]{babel}
\usepackage[T1]{fontenc} 
\usepackage{accents}

\usepackage{hyperref}
\usepackage[noabbrev,capitalize]{cleveref}
\hypersetup{
	colorlinks,
	linkcolor={red!50!black},
	citecolor={blue!50!black},
	urlcolor={blue!80!black}
}
\usepackage{adjustbox}
\usepackage{url}
\usepackage[hmargin=3.7cm,vmargin=3cm]{geometry} 
\usepackage{graphicx}                         
\usepackage{amsmath}                     
\usepackage{amssymb} 
\usepackage{mathrsfs}
\usepackage{stmaryrd}
\usepackage{mathtools}
\usepackage[shortlabels]{enumitem}
\usepackage{tikz-cd}
\usetikzlibrary{babel}
\usepackage{euscript}
\renewcommand{\mathcal}{\EuScript}
\theoremstyle{plain}                                              \makeatletter
\def\swappedhead#1#2#3{%
	\thmnumber{\@upn{\the\thm@headfont#2\@ifnotempty{#1}{.~}}}%
	\thmname{#1}%
	\thmnote{ {\the\thm@notefont(#3)}}}
\makeatother
\newtheorem{thm}{Theorem}[section]

\newtheorem{lem}[thm]{Lemma}
\newtheorem{prop}[thm]{Proposition}
\newtheorem{cor}[thm]{Corollary}
\newtheorem{conj}[thm]{Conjecture}
\theoremstyle{definition}

\newtheorem{rem}[thm]{Remark}
\newtheorem{defn}[thm]{Definition}

\title[Top weight cohomology and handlebodies]{Top weight cohomology of moduli spaces of Riemann surfaces and handlebodies}

\newcommand{\M}{\mathcal{M}}
\newcommand{\HM}{\mathcal{HM}}
\newcommand{\Mbar}{\overline{\M}}
\newcommand{\trop}{\mathrm{trop}}
\newcommand{\R}{\mathbf{R}}
\newcommand{\Z}{\mathbf{Z}}
\newcommand{\C}{\mathbf C}
\newcommand{\Q}{\mathbf Q}
\newcommand{\N}{\mathbf N}
\newcommand{\CV}{\mathcal{CV}}

\newcommand{\Cut}{\operatorname{Cut}}
\newcommand{\Contr}{\operatorname{Contr}}
\newcommand{\DualGraph}{\Gamma}
\newcommand{\corollas}{\gamma}
\newcommand{\geodesic}{\gamma}

\newcommand{\Oper}{\mathcal{O}}

\DeclareMathOperator*{\colim}{colim}
\DeclareMathOperator*{\hocolim}{hocolim}

\newcommand{\Mod}{\mathrm{Mod}}
\newcommand{\HMod}{\mathrm{HMod}}
\newcommand{\Diff}{\mathrm{Diff}}
\DeclareMathOperator{\vcd}{vcd}
\begin{document}

\author{Louis Hainaut}
\author{Dan Petersen}

 \maketitle

\begin{abstract}
    We show that a certain locus inside the moduli space $\M_g$ of hyperbolic surfaces, given by surfaces with ``sufficiently many'' short geodesics, is a classifying space of the handlebody mapping class group. A consequence of the construction is that the top weight cohomology of $\M_g$, studied by Chan--Galatius--Payne, maps injectively into the cohomology of the handlebody mapping class group. 
\end{abstract}

\section{Introduction}

\subsection{The handlebody group and the mapping class group}
\label{subsec: first}
Let $V_{g,n}^m$ be a genus $g$ handlebody, with $m$ distinct ordered marked disks on its boundary and $n$ distinct ordered marked points {also on its boundary}. We let $\Sigma_{g,n}^m = \partial V_{g,n}^m$, a genus $g$ surface with $m$ marked disks and $n$ marked points. We write  $$\Mod_{g,n}^m := \pi_0 \Diff (\Sigma_{g,n}^m)$$
for the \emph{mapping class group} of the surface; here, $\Diff (\Sigma_{g,n}^m)$ denotes the topological group of self-diffeomorphisms of the surface fixing the marked points and disks pointwise, and if $m=0$ we must in addition insist that the diffeomorphism be orientation-preserving. We let similarly
$$\HMod_{g,n}^m := \pi_0 \Diff (V_{g,n}^m)$$
be the \emph{handlebody group}, where similarly $\Diff(V_{g,n}^m)$ denotes self-diffeomorphisms of the surface fixing the marked points and disks pointwise. When $m$ or $n$ vanishes, we may omit it from the notation. For many purposes there is no large difference between considering marked disks or marked points, as the two types of groups are related by the short exact sequence
$$ 0 \to \Z \to \HMod_{g,n}^{m+1} \to \HMod_{g,n+1}^m \to 1, $$
and similarly for the usual mapping class group. 

Any self-diffeomorphism of $V_{g,n}^m$ restricts to a self-diffeomorphism of $\partial V_{g,n}^m=\Sigma_{g,n}^m$. Restriction defines a map $\Diff (V_{g,n}^m) \to \Diff (\Sigma_{g,n}^m)$, and hence 
$$ \HMod_{g,n}^m \to \Mod_{g,n}^m.$$
This map turns out to be \emph{injective}, and one speaks of the \emph{handlebody subgroup} of the mapping class group. This subgroup is well-defined up to conjugation, and for $g=0$ this map is an isomorphism. 

The main theorem of this paper, \cref{mainthm}, is a construction of a classifying space of the handlebody group $\HMod_{g,n}^m$ in terms of Teichm\"uller theory and hyperbolic geometry. But before stating the theorem, we will explain how we were led to the construction, and some consequences.

The homology of the mapping class group has been intensely studied for a long time. The handlebody group is less well studied, but there are many interesting parallels (and differences) between the homologies of the two families of groups. We summarize a few of these in the following table:

\begin{table}[h]
\begin{tabular}{l|l|l}
 & $\Mod_{g,n}^m$ &  $\HMod_{g,n}^m$  \\ \hline
$\vcd$  & $4g-4+m+2n$  &  $4g-4+m+2n$  \\
 $\chi^{\mathrm{orb}}$& $\zeta(1-2g)$ & 0   \\
 homological stability & yes, slope $\tfrac 2 3$ & yes, slope $\tfrac 1 2$  \\
 stable rational cohomology & $\Q[\kappa_1,\kappa_2,\kappa_3,\ldots]$ & $\Q[\kappa_2,\kappa_4,\kappa_6,\ldots]$ \\ 
 infinite loop space & $\Omega^\infty_0 \mathrm{MTSO}(2)$ & $\Omega^\infty_0 \Sigma^\infty_+ B \mathrm{SO}(3)$
\end{tabular}
\end{table}

Let us briefly elaborate on each of these entries in turn:\begin{enumerate}
    \item Harer \cite{harer-vcd} has shown that for $2g-2+m+n>0$ one has $\vcd \Mod_{g,n}^m = 4g-4+m+2n - \delta_{m+n,0} + \delta_{g,0}$, where $\delta$ is a Kronecker delta.  Hirose \cite{hirose-vcd} showed that $\vcd \HMod_g = 4g-5$ for $g\geq 2$, and in fact the virtual cohomological dimensions agree in all cases\footnote{Indeed, the case of $g \geq 2$ and $r,s>0$ follows inductively from this and the Birman exact sequence, cf.~\cite[Theorem 5.5]{bieri}. When $g=0$ the map $ \HMod_{0,m}^n \to \Mod_{0,m}^n$ is an isomorphism. When $g=1$ one has $\HMod_{1,1} \cong \{\pm 1 \} \ltimes \Z$, so $\vcd \HMod_{1,1} = \vcd \Mod_{1,1} = 1$, and again the general result follows from this and the Birman exact sequence. }.
    \item Harer and Zagier \cite{harer-zagier} showed that $\chi^{\mathrm{orb}}(\Mod_{g,1})= \zeta(1-2g)$ for $g \geq 1$. The orbifold Euler characteristics for other values of $m$ and $n$ follow easily from this. On the other hand, Hirose \cite{hirose-vcd} showed that $\chi^{\mathrm{orb}}(\HMod_g)=0$. 
    \item Harer \cite{harer-stability} has proved that $H_i(\Mod_{g}^1,\Z)\to H_{i}(\Mod_{g+1}^1,\Z)$ is an isomorphism for $i \leq \tfrac{2g}{3}$. The range given here is due to Boldsen \cite{boldsen}. The analogous homological stability for handlebody groups is due to Hatcher--Wahl \cite{hatcher-wahl}. The stable range for the mapping class group is known to be optimal, but this is not known for the handlebody group. 
\item The rational cohomology of $\Mod_g^1$ in the range covered by Harer's stability theorem was proved by Madsen--Weiss \cite{madsen-weiss} to be a polynomial algebra on the $\kappa$-classes, confirming the so-called \emph{Mumford conjecture. } Hatcher \cite{hatcher} has outlined a proof that the stable cohomology of the handlebody group is freely generated by the \emph{even} $\kappa$-classes, but a detailed argument is not yet in the literature. Shaul Barkan and Jan Steinebrunner (pers.\ comm.) have an alternative argument, which will appear in forthcoming work and is based on Giansiracusa's theorem \cite{giansiracusa-handlebodies}. 
\item The actual proof of the Mumford conjecture, as well as its analogue for handlebodies, proceeds by identifying the homotopy type of the classifying space of the stable mapping class group (handlebody group) with an explicit infinite loop space, up to plus-construction, as indicated in the table.
\end{enumerate} 


Let us also mention that both $\Diff(\Sigma_{g,n}^m)$ and $\Diff(V_{g,n}^m)$ have contractible connected components, with a few low-genus exceptions. In the surface case this goes back to Earle--Eells \cite{ee}. In the case of handlebodies, isotopy extension shows that $\Diff(V_{g,n}^m)\to \Diff(\Sigma_{g,n}^m)$ is a fibration, and by \cite{hatcher2} the fibers are empty or contractible. (In particular, $\HMod_{g,n}^m\to\Mod_{g,n}^m$ is indeed injective, as stated earlier.) 

We refer to \cite{hensel} for a useful survey of the handlebody group.

\subsection{Top weight cohomology}\label{subsec: top weight}
The Mumford conjecture implies that the stable cohomology of the mapping class group grows like the number-theoretic partition function as a function of $g$, in particular sub-exponentially. On the other hand, Harer--Zagier \cite{harer-zagier} leveraged their calculation of the orbifold Euler characteristic to obtain a formula also for the actual integer-valued Euler characteristic, and they showed that the two are asymptotically equal. Since the Bernoulli numbers grow super-exponentially, it follows that the vast majority of the rational homology of the mapping class group lives in the unstable range and is unaccounted for by the Mumford conjecture. 

For a long time this was a mildly unsatisfactory state of affairs: it was known that most homology would be unstable as $g \gg 0$, but the only explicitly constructed unstable homology classes were isolated examples in low genus. This recently changed with work of Chan--Galatius--Payne \cite{changalatiuspayne}, who studied the \emph{top weight cohomology of the moduli space of curves}. 

Let $\M_{g,n}$ denote the moduli orbifold of $n$-pointed Riemann surfaces. As an orbifold, one has a homotopy equivalence $\M_{g,n} \simeq B\Mod_{g,n}$, and $H_\bullet(\M_{g,n},\Z)\cong H_\bullet(\Mod_{g,n},\Z).$ Being a complex orbifold of dimension $3g-3+n$, it satisfies Poincar\'e duality with $\Q$-coefficients:
$$ H_k(\M_{g,n},\Q) \cong H^{6g-6+2n-k}_c(\M_{g,n},\Q).$$
The space $\M_{g,n}$ is not just a complex orbifold, but it is the analytification of an algebraic Deligne--Mumford stack. This implies in particular that its rational cohomology carries a natural \emph{mixed Hodge structure}, after Deligne. By the ``top weight cohomology'' we mean the weight zero\footnote{The terminology ``top weight'' may seem like a misnomer, since weight zero is the \emph{lowest} possible weight. The mismatch occurs because Chan--Galatius--Payne consider usual cohomology.} part of the mixed Hodge structure on compactly supported cohomology, $W_0 H^\bullet_c(\M_{g,n},\Q)$.

If $U$ is a smooth variety, then the top weight cohomology of $U$ can be computed by making explicit Deligne's construction of the mixed Hodge structure on smooth varieties \cite{hodge2}. Firstly, choose a compactification $U\subset \overline U$ such that $\overline U \setminus U$ is a simple normal crossing divisor, and consider the divisor as defining a stratification of $\overline U$. Secondly, write down the cochain complex $D(U,\overline U)$ which in degree $p$ is the $\Q$-vector space with a basis indexed by the set of strata of codimension $p$, with an evident differential recording adjacencies of boundary strata. The cohomology of $D(U,\overline U)$ is independent of the choice of $\overline U$, and coincides with $W_0 H^\bullet_c(U,\Q)$. The top weight cohomology is in a sense the most combinatorial part of the cohomology; in the example just explained, the top weight part simply records the combinatorics of how the boundary divisors in a compactification intersect. 

In the case of $\M_{g,n}$, there is a god-given choice of normal crossing divisor compactification, namely the \emph{Deligne--Mumford compactification} $\Mbar_{g,n}$. The concrete description of the top weight cohomology given in the preceding paragraph carries over also in this case, although with some additional subtleties when $g>0$ due to the fact that the boundary is not a \emph{simple} normal crossing divisor. The boundary strata of $\Mbar_{g,n}$ index the possible topological types of an $n$-pointed stable curve of genus $g$, and are indexed by \emph{stable graphs}, as in \cref{fig: degen1} and \cref{fig: degen2}.
\begin{figure}[ht]
    \centering
\includegraphics[width=12cm]{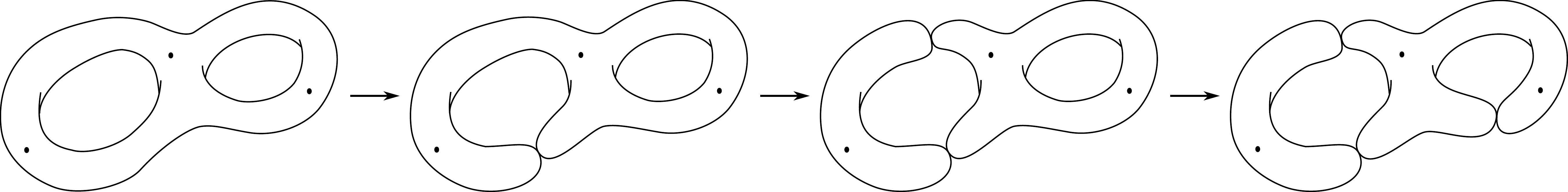}
    \caption{A sequence of degenerations of stable curves inside $\Mbar_{2,3}$.}
    \label{fig: degen1}
\end{figure}

\begin{figure}[ht]
    \centering
\includegraphics[width=10cm]{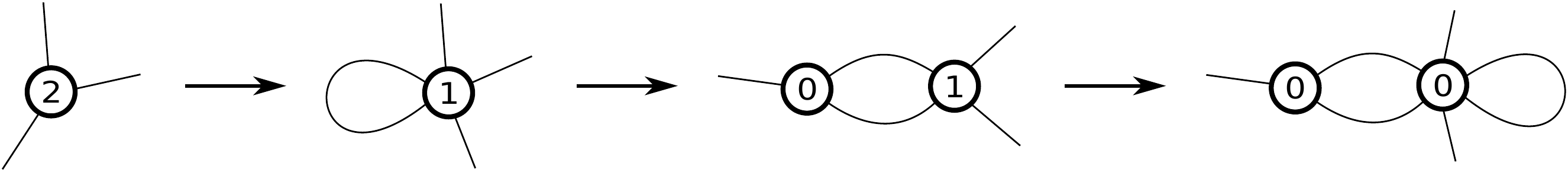}
    \caption{The corresponding sequence of stable graphs.}
    \label{fig: degen2}
\end{figure}

Two stable graphs define adjacent strata precisely when one can be obtained from the other by contracting a single edge, as in the figure. It follows that the cochain complex $D(\M_{g,n},\Mbar_{g,n})$ is an instance of a \emph{graph complex}: it is a cochain complex spanned by isomorphism classes of graphs of a certain type, and a differential given by edge expansions, with signs added to ensure $d^2=0$. In fact, $D(\M_{g,n},\Mbar_{g,n})$ is a close relative of the \emph{hairy graph complex} $\mathsf{HGC}_0$, going back to Kontsevich. The complex $\mathsf{HGC}_0$ splits into summands $\mathsf{HGC}_0^{(g,n)}$ indexed by \emph{loop order} $g$ and the number of \emph{hairs} $n$, and $\mathsf{HGC}_0^{(g,n)}$ is naturally a subcomplex of $D(\M_{g,n},\Mbar_{g,n})$ --- the graphs which span $D(\M_{g,n},\Mbar_{g,n})$ have genus decorations at the vertices, and $\mathsf{HGC}_0^{(g,n)}$ is the subcomplex spanned by graphs with all vertices of genus zero. The theorem of Chan--Galatius--Payne is that in fact $\mathsf{HGC}_0^{(g,n)}\hookrightarrow D(\M_{g,n},\Mbar_{g,n})$ is (nearly) always a quasi-isomorphism.  

\begin{thm}[Chan--Galatius--Payne]The map 
$H^\bullet(\mathsf{HGC}_0^{(g,n)}) \to W_0 H^\bullet_c(\M_{g,n},\Q)$ is an isomorphism for all $(g,n) \neq (1,1)$. 
\end{thm}

Graph complexes like $\mathsf{HGC}$ and its many relatives are ubiquitous in  low-dimensional and high-dimensional topology, moduli theory, deformation quantization, and many other areas. We will not attempt to survey this sprawling subject, but mention \cite{willwachericm} as a jumping-in point. 

Willwacher \cite{willwacher} calculated $H^0(\mathsf{GC}_2)$, and showed that it can be identified with the \emph{Grothen\-dieck--Teich\-m\"uller Lie algebra} $\mathfrak{grt}_1$, introduced by Drinfeld \cite{drinfeld}. Here $\mathsf{GC}$ denotes the summand of $\mathsf{HGC}$ of graphs without hairs, and the subscript $2$ denotes a shift in degrees.\footnote{More generally one often considers a whole family of complexes $\mathsf{HGC}_{m,n}$ for $m,n\in\Z$; up to degree shifts their cohomologies depend only on the parity of $m$ and $n$. But we also point out that these complexes are conventionally defined by symmetrizing or antisymmetrizing with respect to permutations of the set of legs, which we do not want to do.} Combining Willwacher's theorem and the preceding corollary, and keeping track of the shift involved, one deduces the following. 

\begin{cor}\label{corollary grt}
There is an isomorphism $\bigoplus_{g \geq 2} W_0 H^{2g}_c(\M_g,\Q) \cong \mathfrak{grt}_1$.
\end{cor}

There is a map $\mathrm{FreeLie}(\sigma_3,\sigma_5,\sigma_7,\ldots)\to \mathfrak{grt}_1 $ which is conjectured to be an isomorphism \cite[p.~859]{drinfeld}. Under the isomorphism of \cref{corollary grt}, the generator $\sigma_p$ goes to a class in genus $p$, and the Lie bracket is compatible with the grading by genus. It follows from a theorem of Brown \cite{brown} that $\mathrm{FreeLie}(\sigma_3,\sigma_5,\sigma_7,\ldots)\to \mathfrak{grt}_1 $ is injective. From this fact, and an estimate for the dimensions of the graded pieces of the free Lie algebra, Chan--Galatius--Payne deduce:

\begin{cor}The Betti numbers $\dim H_{4g-6}(\Mod_g,\Q)$ are nonzero for $g=3,5$ and $g \geq 7$, and grow at least exponentially in $g$.
\end{cor}

Note that $\vcd \Mod_g = 4g-5$. It is known that $H_{4g-5}(\Mod_g,\Q) =0$ for all $g$; this result was announced by Harer, but a proof did not appear until \cite{cfp,mss}. Hence the construction produces nontrivial ``maximally unstable'' rational homology classes (i.e.\ classes in the highest degree they could possibly appear), and lots of them. Combining the aforementioned results on weight $0$ with subsequent work on weight $2$ \cite{paynewillwacher} and weight $11$ \cite{pw11} (which is the smallest nonzero odd weight \cite{bfp}), there are now $\approx 50$ different values of $k$ for which we know that $\dim H^{2g+k}_c(\M_g,\Q)$ grows at least exponentially with $g$ \cite[Corollary 1.2]{pw11}.

\subsection{Moduli spaces of tropical curves and handlebodies}
\label{subsec: moduli of tropical curves}

The work of Chan--Galatius--Payne can also be understood in terms of \emph{moduli spaces of tropical curves} \cite{bmv,acp,cap}. We define a \emph{tropical curve} of genus $g$ with $n$ markings to be a stable graph of genus $g$ with $n$ external legs, exactly as those in \cref{fig: degen2}, together with a positive real number attached to each internal edge, which we call the \emph{length} of the edge. The moduli space $\M_{g,n}^\trop$ of such tropical curves is an orbi-cone complex: for a given fixed stable graph $\DualGraph$ with $p$ internal edges, the subspace of $\M_{g,n}^\trop$ of tropical curves with underlying graph $\DualGraph$ is the quotient orbifold $[(\R_{>0})^p / \mathrm{Aut}(\DualGraph)]$; these orbi-cells are glued together in such a way that letting an edge length go to zero corresponds to contracting the corresponding edge. Thus $p$-dimensional cells of $\M_{g,n}^\trop$ are in bijection with codimension $p$ strata of $\Mbar_{g,n}$. 

The space $\M_{g,n}^\trop$ is contractible: letting edge lengths go to zero contracts everything down to the unique stable graph without edges. But it still has interesting topology --- in particular, its cohomology with compact support is highly nontrivial. Indeed, in the same way that usual cohomology can be computed from a CW decomposition, cohomology with compact support can be computed from a \emph{stratification} into open cells. Applying this to the stratification by stable graphs, one finds\footnote{We need $\Q$-coefficients here, since we only have an orbi-cell stratification, and integrally these orbicells may have very nontrivial (co)homology.} that the complex of cellular chains $C^\bullet_c(\M_{g,n}^\trop,\Q)$ is nothing but the complex $ D(\M_{g,n},\Mbar_{g,n})$ described in the previous subsection. In particular, $H^\bullet_c(\M_{g,n}^\trop,\Q) \cong W_0 H^\bullet_c(\M_{g,n},\Q)$.

We define $\CV_{g,n} \subset \M_{g,n}^\trop$, the \emph{Culler--Vogtmann moduli space of graphs}, to be the union of those orbicells corresponding to stable graphs with all vertex decorations $0$. Then $\CV_{g,n}$ is an open subspace of $\M_{g,n}^\trop$. These spaces were introduced in the study of automorphisms of free groups  \cite{cv}: one has $\CV_{g,1} \simeq B\mathrm{Aut}(F_g)$ and $\CV_{g,0} \simeq B\mathrm{Out}(F_g)$. Reasoning as in the previous paragraph, one sees that $C^\bullet_c(\CV_{g,n},\Q) \cong \mathsf{HGC}_0^{(g,n)}$. Moreover, the map from hairy graph cohomology to top weight cohomology of $\M_{g,n}$ can be identified with the natural map
$$ H^\bullet_c(\CV_{g,n},\Q) \to H^\bullet_c(\M_{g,n}^\trop,\Q)$$
obtained from $H^\bullet_c(-)$ being covariantly functorial for open immersions. 

In their paper, Chan--Galatius--Payne also give an alternative geometric interpretation of the map $$H^\bullet(\mathsf{HGC}_0^{(g,n)}) \to H^\bullet_c(\M_{g,n},\Q),$$ by means of a ``tropicalization'' morphism $\lambda_{g,n} : \M_{g,n}\to\M_{g,n}^\trop$. This interpretation will be an important ingredient in this paper, so let us try to explain it.

Here is their construction: interpret $\M_{g,n}$ as the space of hyperbolic surfaces of genus $g$ with $n$ cusps. Choose a real number $\varepsilon$ small enough that any two closed geodesics on such a surface of length $<\varepsilon$ are disjoint; we call such geodesics \emph{short}. For a hyperbolic surface $\Sigma \in \M_{g,n}$, consider the nodal surface obtained from $\Sigma$ by contracting each short geodesic to a point. This nodal surface defines a stable graph, whose internal edges are in bijection with the short geodesics on $\Sigma$, and we decorate each edge with the real number $-\log(\ell/\varepsilon)$, where $\ell$ is the length of the geodesic. The result is a point of $\M_{g,n}^\trop$, and we declare this point to be $\lambda_{g,n}(\Sigma).$ This turns out to be a continuous, proper map, and since $H^\bullet_c(-)$ is contravariantly functorial for proper morphisms one obtains a homomorphism
$$ H^\bullet_c(\M_{g,n}^\trop,\Q) \to H^\bullet_c(\M_{g,n},\Q),$$
which turns out to agree with the composition $$ H^\bullet_c(\M_{g,n}^\trop,\Q) \cong W_0 H^\bullet_c(\M_{g,n},\Q) \subset H^\bullet_c(\M_{g,n},\Q).$$

In summary, we have the following diagram of spaces:
\[\begin{tikzcd}
    & \M_{g,n} \arrow[d, two heads, "\text{proper}"]\\
    \CV_{g,n} \arrow[r,"\text{open}", hook] & \M_{g,n}^\trop
\end{tikzcd}\]
Applying the functor $H^\bullet_c(-)$, and using that it is covariant for open embeddings and contravariant for proper maps, defines the map from hairy graph cohomology to the compact support cohomology of $\M_{g,n}$. But there is now an obvious way to fill in the diagram, and obtain an alternative factorization of the map
$ H^\bullet_c(\CV_{g,n},\Q) \to H^\bullet_c(\M_{g,n},\Q)$. Let us define $\HM_{g,n} \subset \M_{g,n}$ to be the inverse image of $\CV_{g,n}$ under the tropicalization morphism $\lambda_{g,n}$:
\[\begin{tikzcd}
    \HM_{g,n} \arrow[d, two heads, "\text{proper}"]\arrow[r,"\text{open}", hook] & \M_{g,n} \arrow[d, two heads, "\text{proper}"]\\
    \CV_{g,n} \arrow[r,"\text{open}", hook] & \M_{g,n}^\trop
\end{tikzcd}\]
Explicitly, $\HM_{g,n}$ is the open subspace of $\M_{g,n}$ parametrizing hyperbolic surfaces with the following property: \emph{cutting the surface along the union of all its short geodesics decomposes the surface into pieces of genus zero}. More generally, let us define $\HM_{g,n}^m$ for $2g-2+n+m>0$ to be the space of hyperbolic surfaces of genus $g$ with $n$ cusps and $m$ parametrized geodesic boundary components of length $\varepsilon$, still with the property that cutting the surface along short geodesics decomposes it into pieces of genus zero. We can now state our main theorem. 

\begin{thm}\label{mainthm}
The space $\HM_{g,n}^m$ is a classifying space for the handlebody group $\HMod_{g,n}^m$. If $m=0$ this must be understood in the orbifold sense, but if $m>0$ then $\HM_{g,n}^m$ is an actual topological space.
\end{thm}
 
We expect \cref{mainthm} to be broadly useful in the study of the handlebody group. Indeed, there is little doubt that Teichm\"uller theory has been a very powerful tool in the study of mapping class groups of surfaces, and one may hope that the model developed here can be similarly powerful in the study of handlebodies. In fact, no natural geometric model for $B\HMod_{g,n}^m$ has been known until now. But the immediate consequence of \cref{mainthm} is the fact that the map from the cohomology of the hairy graph complex to the homology of the mapping class group factors through the homology of the handlebody subgroup, using the composition
$$ H^\bullet(\mathsf{HGC}_0^{(g,n)}) \cong H^\bullet_c(\CV_{g,n},\Q) \to H^\bullet_c(\HM_{g,n},\Q) \to H^\bullet_c(\M_{g,n},\Q).$$Combined with Poincar\'e duality $H_k(\HM_{g,n},\Q) \cong H^{6g-6+2n-k}_c(\HM_{g,n},\Q)$ this immediately implies:

\begin{cor}
There is an injection $\mathfrak{grt}_1 \hookrightarrow \bigoplus_{g \geq 2} H_{4g-6}(\HMod_g,\Q)$. The composite $\mathrm{FreeLie}(\sigma_3,\sigma_5,\ldots) \hookrightarrow \bigoplus_{g \geq 2} H_{4g-6}(\HMod_g,\Q)$ takes the class $\sigma_p$ to a homology class in genus $p$, and the Lie bracket is compatible with the grading by genus. 
\end{cor}

\begin{cor}
The group $H_{4g-6}(\HMod_g,\Q)$ is nonzero for $g=3$, $g=5$ and $g \geq 7$. The Betti numbers grow at least exponentially: more specifically, $$\liminf_{g \to\infty} \beta^{-g} \dim H_{4g-6}(\HMod_g,\Q) \geq 1,$$ where $\beta = 1.3247...$ is the real root of $t^3-t-1=0$.\end{cor}

Further known calculations of graph cohomology give yet more nontrivial classes in the homology of the handlebody groups: 
\begin{align*}
    H_{15}(\HMod_6,\Q)&\neq 0, \\
    H_{23}(\HMod_8,\Q)&\neq 0, \\
    H_{27}(\HMod_{9},\Q)&\neq 0, \\
    H_{27}(\HMod_{10},\Q)&\neq 0, \\
    \dim H_{31}(\HMod_{10},\Q) &\geq 2,
\end{align*}
where we have used computations quoted from \cite{willwachericm}. A large number of computations for $n>0$ are also in the literature 
\cite{hairy,cgp2,claudia,gadish,gadish-hainaut}. 

\begin{rem}
The tropicalization map $\HM_{g} \to \CV_{g}$ realizes geometrically the natural homomorphism from the handlebody group to $\mathrm{Out}(F_g)$. More generally, the natural map $\Diff(V_{g,n}) \to \mathrm{hAut}(V_{g,n})$ from diffeomorphisms to homotopy automorphisms is realized on classifying spaces by the map $\HM_{g,n}\to\CV_{g,n}$. Note that $\mathrm{hAut}(V_{g,n}) \simeq \pi_0 \mathrm{hAut}(V_{g,n}) \cong \Gamma_{g,n}$, where $\Gamma_{g,n}$ denotes the family of groups considered e.g.\ in \cite{chkv}. 
\end{rem}

\begin{rem}
    An important ingredient in our proof of \cref{mainthm} is a theorem of Giansiracusa, that the topological modular operad $\{B \Diff (V_g^m) \}$ is the derived modular envelope of the framed little disk operad. Now, the framed little disk operad is \emph{motivic} --- see \cite{vaintrob,vaintrob2}, or the review of Vaintrob's work in \cite[Section 8]{bdpw} --- and then the handlebody operad, too, must be equally motivic. In particular, the chains on the framed little disk operad have a natural ind-mixed Hodge structure, and a consequence is that the homology of the handlebody groups themselves carries a natural (Tate type) mixed Hodge structure. With respect to this mixed Hodge structure, our theorem could be more naturally formulated as the assertion that the top weight cohomology of the handlebody group coincides with the top weight cohomology of the whole mapping class group. However, we will not develop this perspective further here.  
\end{rem}

\begin{rem}
    It would be possible to prove that top weight cohomology of $\Mod_{g,n}^m$ injects into the cohomology of $\HMod_{g,n}^m$ without passing through the geometric \cref{mainthm}. (But the geometric model for the classifying space seems independently interesting.) Namely, Giansiracusa's theorem allows one to write down a model of $B \HMod_g^m$ as a homotopy colimit indexed over certain stable graphs, as explained immediately after \cref{thm: giansiracusa}. From this one may derive a graph complex computing the cohomology of $\HMod_g^m$, by reasoning as in \cite[Section 7]{giansiracusa-handlebodies}. Giansiracusa works with the diagram of spaces $\Oper : F_{g,n} \to \mathsf{Spaces}$ (notation as in \cref{giansiracusa section}), which leads to a somewhat unwieldy graph complex with bivalent vertices; applying the same reasoning to $\Oper^\# : F_{g,n}^{\mathrm{stab}} \to \mathsf{Spaces}$ gives a graph complex with stable graphs, which on $H_0$ recovers a variant of the hairy graph complex. Giansiracusa uses that $\Oper$ is a formal functor \cite{formality}; the same is true for $\Oper^\#$ by motivic considerations \cite{cirici-horel,vaintrob}. Although we do not follow this route in this paper, an advantage of this approach would be that it would just as well produce a graph complex of sorts computing weight two compactly supported cohomology of $\HM_{g,n}$ (or higher weights, for that matter). It would be interesting to compare it with the weight two compactly supported cohomology of $\M_{g,n}$, studied by Payne--Willwacher \cite{paynewillwacher}.
\end{rem}

\begin{rem}As we noted in Subsection \ref{subsec: top weight}, one could deduce from Euler characteristic considerations that the unstable homology of $\Mod_{g,n}$ would be significantly bigger than the stable homology, even before the work of Chan--Galatius--Payne. The same is not true for $\HMod_{g,n}$: as noted in Subsection \ref{subsec: first}, the orbifold Euler characteristic of $\HMod_{g,n}$ vanishes, and prior to this paper it would presumably have been a possibility that ``most'' of the homology of the handlebody group would lie in the stable range. Nevertheless, just as in the case of $\Mod_{g,n}$, we expect that even the exponentially growing family of classes detected from top weight cohomology make up merely a minuscule fragment of the whole unstable homology as $g \gg 0$. 
\end{rem}

\begin{rem}\label{rem: conjecture}
    We were originally led to thinking about handlebodies in connection with the top weight cohomology of moduli spaces of curves for completely independent reasons. If $\pi : \M_{g,n} \to \M_g$ denotes the evident forgetful map, then one may consider the weight zero part of the corresponding compactly supported Leray--Serre spectral sequence
    $$ E_2^{pq} = \operatorname{Gr}^W_0 H^p_c(\M_g,R^q \pi_! \Q) \implies \operatorname{Gr}^W_0 H^{p+q}_c(\M_{g,n},\Q). $$
    However, as was pointed out to us by Nir Gadish, there is yet another spectral sequence which one may also use to study the top weight cohomology of $\M_{g,n}$, which is the subject of forthcoming work of Bibby--Chan--Gadish--Yun \cite{bibby}. Namely, one may consider the compactly supported Leray--Serre spectral sequence of the projection $\rho : \CV_{g,n} \to \CV_g$, 
    $$ E_2^{pq} = H^p_c(\CV_g,R^q \rho_! \Q) \implies H^{p+q}_c(\CV_{g,n},\Q). $$
    We conjecture that not only are the abutments naturally isomorphic (by \cite{changalatiuspayne}), but that the whole spectral sequences can be identified in a natural manner. {In the final section (\S \ref{section Leray Serre}) of the paper we try to motivate this conjecture, and explain why it leads to thinking about thickening surfaces to handlebodies.}
    
\end{rem}

\textbf{Acknowledgements.} We are grateful to Nir Gadish, from whom we learned the idea of studying top weight cohomology by means of the compactly supported Leray--Serre spectral sequence for $\CV_{g,n}\to\CV_g$. We thank Sam Payne and Jan Steinebrunner for useful comments. Both authors were supported by the grant ERC-2017-STG 759082 and a Wallenberg Academy Fellowship.


\section{Some hyperbolic geometry and Teichm\"uller theory}

\subsection{Fenchel--Nielsen coordinates and the Bers atlas}
Let $2g-2+n>0$. One can define $\M_{g,n}$ to be the moduli space of $n$-pointed Riemann surfaces, but in this section $\M_{g,n}$ will be thought of as the space parametrizing hyperbolic surfaces of genus $g$ with $n$ ordered cusps. 
	Choose an oriented reference surface $S$ of genus $g$ with $n$ punctures. This choice gives rise to a choice of universal cover of the moduli space $\M_{g,n}$; we obtain the Teichm\"uller space $T(S)$ parametrizing isotopy classes of hyperbolic metrics on $S$ with cusps along the punctures, and $T(S)/\Mod_{g,n} \cong \M_{g,n}$. If $\geodesic$ is a simple closed curve on $S$, then for every hyperbolic metric on $S$ there is a unique geodesic in the free homotopy class of $\geodesic$. The length of this geodesic depends only on the isotopy class of the hyperbolic metric, and defines a continuous function $\ell_\geodesic : T(S) \to \R_{>0}$. 
	
	The length functions $\ell_\geodesic$ make up ``half'' of the Fenchel--Nielsen coordinate system on Teichm\"uller space. Choose a pants decomposition of $S$; it will consist of disjoint simple closed curves $\geodesic_1,\ldots,\geodesic_{3g-3+n}$. Now there is always a unique hyperbolic pair of pants with specified lengths of its three boundaries (including the case of a cusp, which we consider as a boundary of length zero). Thus from the numbers $\ell_{\geodesic_1},\ldots,\ell_{\geodesic_{3g-3+n}}$ we can uniquely reconstruct a hyperbolic metric on each pair of pants in the pants decomposition. In order to pin down a hyperbolic metric on $S$, all that remains is to specify a ``twist parameter'' along each $\geodesic_i$, giving an orientation for how to glue the pairs of pants together. A famously confusing fact is that the twist parameter does not take values in $S^1$, but in $\R$. One way to think about this mismatch is that Fenchel--Nielsen coordinates as explained here actually define coordinates on the quotient $T(S)/\Z^{3g-3+n}$, where $\Z^{3g-3+n} \subset \Mod_{g,n}$ denotes the subgroup generated by Dehn twists along all curves in the chosen pants decomposition of $S$, and then we have 
	$$ T(S)/\Z^{3g-3+n} \cong \R_{>0}^{3g-3+n} \times (S^1)^{3g-3+n}.$$

Fenchel--Nielsen coordinates may be considered either as global coordinates on Teichm\"uller space, or as local coordinates around any point of $\M_{g,n}$ (since $T(S)\to \M_{g,n}$ is a covering map). By a small modification of the construction, one can also give local coordinates around any point of the Deligne--Mumford compactification $\Mbar_{g,n}$: this is called the \emph{Bers atlas}, see \cite{ebert-giansiracusa} for a review. Unlike the usual Fenchel--Nielsen coordinates there is no corresponding global coordinate system --- the space $\Mbar_{g,n}$ is simply connected and admits no nontrivial covering space. Again, fix a pants decomposition of the reference surface $S$, and consider the coordinates 
$$ T(S)/\Z^{3g-3+n} \cong \R_{>0}^{3g-3+n} \times (S^1)^{3g-3+n}$$
described in the previous paragraph. Identify $\R_{>0}\times S^1$ with the punctured plane $\C^\times$. By gluing in the origin, we obtain an open embedding 
$$ T(S)/\Z^{3g-3+n} \subset \C^{3g-3+n}.$$
Now the covering map $T(S)/\Z^{3g-3+n} \to \M_{g,n}$ extends to a local diffeomorphism $\C^{3g-3+n} \to \Mbar_{g,n}$. Informally, we are allowing the lengths of the curves in the pants decomposition to become zero, in which case they are interpreted as cusps of the hyperbolic metric, and the twist parameter becomes undefined. The map $\C^{3g-3+n}\to \Mbar_{g,n}$ interprets the result as a possibly nodal surface, defining a point of $\Mbar_{g,n}$. Note that $\C^{3g-3+n}\to \Mbar_{g,n}$ is not surjective, since the image contains only those topological types of nodal surface which can be obtained by contracting a subset of the curves in the chosen pants decomposition. But taking the union over all combinatorial types of pants decomposition, we obtain a coordinate system around any point of the Deligne--Mumford compactification.

For any simple closed curve $\geodesic$ on $S$, the length function $\ell_\geodesic$ extends continuously to a map $\overline{\ell_\geodesic}: \C^{3g-3+n} \to [0,\infty]$, where $\C^{3g-3+n}$ is the space considered in the previous paragraph. Under this extended length function, a curve of length zero corresponds precisely to a curve that gets contracted to a cusp, and a curve of infinite length corresponds to a curve that passes through a curve that gets contracted to a cusp.

Suppose that $2g-2+n+m>0$. If we take instead $S$ to be a genus $g$ surface with $n$ ordered punctures and $m$ ordered boundary components, then we let $T(S)$ be the space isotopy classes of hyperbolic metrics on $S$ with cusps along the punctures, and where each boundary component is totally geodesic of a fixed length $\varepsilon$. The quotient $T(S)/\Mod_{g,n}^m$ can be identified with the space $\M_{g,n}^m$ parametrizing hyperbolic surfaces of genus $g$ with $n$ cusps and $m$ parametrized boundary components. 

The space $T(S)$ can still be parametrized by Fenchel--Nielsen coordinates, just as in the case $m=0$. Then the pants decomposition of $S$ consists of $3g-3+n+2m$ curves, of which $m$ are the chosen boundary curves. Since we fixed the lengths of the boundaries, the $m$ boundary curves do not admit a length parameter in the Fenchel--Nielsen coordinates, but we do include a twist parameter for each boundary. Thus $\M_{g,n}^m$ is of real dimension $6g-6+2n+3m$.


\subsection{Thick Teichm\"uller space}\label{subsec: thick}

The moduli space $\M_{g,n}$ is famously not compact, whence the Deligne--Mumford compactification $\Mbar_{g,n}$. The \emph{thick} locus of the moduli space, $\M_{g,n}(\varepsilon)$, may be thought of as the complement of a tubular neighborhood of the Deligne--Mumford boundary inside $\Mbar_{g,n}$. The thick locus depends on a parameter $\varepsilon$, which must be chosen small enough. Recall that sufficiently short geodesics on a hyperbolic surface are always disjoint:

\begin{thm}[Collar lemma]\label{thm:collar-lemma}
Let $S$ be a hyperbolic surface. If $\geodesic, \geodesic' \subset S$ are closed geodesic curves with $\ell(\geodesic), \, \ell(\geodesic') < \log(3+2\sqrt 2)$, then $\geodesic \cap \geodesic' = \varnothing$. 
\end{thm}

    Fix for the rest of the paper a real number $\varepsilon < \log(3+2\sqrt 2)$. A closed geodesic of length $\leq \varepsilon$ will be called a \emph{short geodesic}. 

\begin{defn}
    The \emph{$\varepsilon$-thick locus} inside $\M_{g,n}^m$ is the sublocus parametrizing hyperbolic surfaces in which all closed curves have length $\ell(\geodesic) \geq \varepsilon$, and each boundary component has length exactly $\varepsilon$. We denote the $\varepsilon$-thick locus by $\M_{g,n}^m(\varepsilon)$. 
\end{defn}

We summarize the properties of $\M_{g,n}^m(\varepsilon)$ in the following theorem.

\begin{thm}The space $\M_{g,n}^m(\varepsilon)$ is a compact manifold-with-corners when $m>0$. When $m=0$ it is merely an orbifold-with-corners. The inclusion $\M_{g,n}^m(\varepsilon) \hookrightarrow \M_{g,n}^m$ is a homotopy equivalence.  
\end{thm}

Compactness of $\M_{g,n}^m(\varepsilon)$ is the Mumford compactness theorem \cite{mumfordcompactness}. The fact that it is a manifold-with-corners is a consequence of the collar lemma and Fenchel--Nielsen coordinates. Indeed, given a hyperbolic surface $S$, we know from the collar lemma that the short geodesics on $S$ are disjoint, which implies in particular that there is a pants decomposition of $S$ which includes all short geodesics. In Fenchel--Nielsen coordinates with respect to this chosen pants decomposition, the equations $\ell_\geodesic[X] \geq \varepsilon$ defining $\M_{g,n}^m(\varepsilon)$ manifestly cut out a manifold with corners. Additionally, we need to know that there is a neighborhood of $S$ inside the moduli space in which there can be no additional short geodesics: this follows from the following lemma of Wolpert. 

\begin{thm}[Wolpert's lemma] For any $[X], [Y] \in T(S)$, and $\geodesic$ a simple closed curve on $S$, one has the inequality
	$$ |\log \ell_\geodesic[X] - \log \ell_\geodesic[Y] | \leq d_{\mathrm{Teich}}(X,Y)$$
	where $d_{\mathrm{Teich}}$ denotes the Teichm\"uller metric. 
\end{thm}

\begin{proof}
    \cite[Lemma 3.1]{wolpert-spectra}. 
\end{proof}

For the fact that $\M_{g,n}^m(\varepsilon) \hookrightarrow \M_{g,n}^m$ is a homotopy equivalence, see \cite{ji-wolpert,Ji-retract-Teichmuller}. In brief, Ji and Wolpert construct a vector field supported on $\M_{g,n}^m \setminus \M_{g,n}^m(\varepsilon)$, such that flowing along this vector field has the effect of continuously increasing the length of the shortest closed geodesic of the hyperbolic surface. If there are multiple shortest geodesics, it increases the length of all of them at the same rate. Flowing along this vector field defines a deformation retraction of $\M_{g,n}^m$ down to the thick locus. 


\section{Giansiracusa's theorem and the hyperbolic model for handlebodies}
\label{giansiracusa section}
\subsection{Cyclic operads and modular operads}

In this section we briefly review the notions of cyclic and modular operad, mostly to fix notation and conventions. For a more thorough treatment we refer to the original sources, see \cite{getzlerkapranov}. The definitions we use are modelled on those of Costello \cite{costelloAinfinityOperad}.

A \emph{dual graph} consists of a finite set $H$ of half-edges, a finite set $V$ of vertices, an involution $i$ on $H$, a function $f:H \to V$, and a genus function $g:V\to \mathbf N$. Fix-points of $i$ are called \emph{legs}, and orbits of size two are called \emph{edges}. We denote $n(v) := |f^{-1}(v)|$, for $v \in V$. 

A dual graph is \emph{stable} if $2g(v)-2+n(v)>0$ for all $v \in V$. It is \emph{semistable} if $2g(v)-2+n(v) \geq 0$ for all $v \in V$. 

There are two natural ways of obtaining from a given dual graph $\DualGraph$ a graph without edges: we may contract each edge, but we may also cut each edge into two legs, obtaining a disjoint union of one-vertex graphs. 

Explicitly, if $\DualGraph = (H,V)$ is a dual graph, then $\Contr \DualGraph = (H',V')$ is defined as follows: $H'$ is the set of legs of $\DualGraph$; $V'$ is the quotient of $V$ by the equivalence relation generated by $f(h) \sim f(i(h))$ for $h \in H$; the genus function $g':V' \to \mathbf N$ is the unique function satisfying
$$ 2g'(v') -2+n(v') = \sum_{\substack{v \in V \\ [v]=v'}} 2g(v)-2+n(v)$$
for all $v' \in V'$. That is, the sum on the right-hand side runs over the preimage of $v'$ under the natural quotient $V \to V'$. Note that if $\DualGraph$ is stable or semistable then so is $\Contr \DualGraph$.

We define $\Cut \DualGraph$ to be the graph with the same half-edges, vertices, and genus decorations as $\DualGraph$, but the involution on the half-edges of $\Cut \DualGraph$ is the identity. 

We define the category $\mathsf{Graphs}$ to be the category whose objects are semistable graphs with no edges. A morphism $\corollas \to \corollas'$ is a semistable graph $\DualGraph$ with $\Cut \DualGraph = \corollas$, together with an isomorphism $\Contr \DualGraph \cong \corollas'$. The category $\mathsf{Graphs}$ becomes symmetric monoidal under the disjoint union of graphs. Let $\mathsf{Forests} \subset \mathsf{Graphs}$ be the subcategory whose morphisms are the graphs with no loops, and all genus decorations are identically zero.

\begin{defn}
	A \emph{modular operad} in a symmetric monoidal category $\mathcal E$ is a strong symmetric monoidal functor $\mathsf{Graphs} \to \mathcal E$. A \emph{cyclic operad} is a strong symmetric monoidal functor $\mathsf{Forests} \to \mathcal E$.
\end{defn}
If $\mathcal O$ is a modular operad, then we sometimes write $\mathcal O(g,n)$ for $\mathcal O(\ast_{g,n})$, where  $\ast_{g,n}$ denotes a graph with a single genus $g$ vertex and $n$ legs. We call $\mathcal O(g,n)$ the \emph{component} of $\mathcal O$ of \emph{genus} $g$ and \emph{arity} $n$. A modular operad can be defined more explicitly as a collection of objects $\mathcal O(g,n)$ in $\mathcal E$ for all $2g-2+n\geq 0$, with an action of $\mathbb S_n$ on $\mathcal O(g,n)$ for all $g,n$, and gluing maps
$$ \mathcal O(g,n+1) \otimes \mathcal O(g',n'+1) \to \mathcal O(g+g',n+n') \qquad \text{and}\qquad \mathcal O(g,n+2) \to \mathcal O(g+1,n)$$
satisfying suitable associativity and equivariance axioms. Similarly, if $\mathcal O$ is a cyclic operad then we sometimes write $\mathcal O(n)$ for $\mathcal O(\ast_{0,n})$.

The examples of greatest interest to us will be operads in the category of topological spaces, i.e.\ \emph{topological modular operads}. The collection of spaces $\{B\Diff(\Sigma_g^n)\}$ admit a structure of topological modular operad, given by gluing surfaces along their boundaries; we call this the \emph{surface operad}. (This fact is slightly more subtle than it first appears, see e.g.\ the discussion in \cite[Section 3.1]{bdpw}.) Similarly, the collection of spaces $\{B\Diff(V_g^n)\}$ also assemble into a topological modular operad, which we call the \emph{handlebody operad}.

The cyclic part of the surface operad is called the \emph{framed little disk operad}. It can also be modeled in terms of gluing of oriented disks, see \cite{getzlerframed}.

The definition of modular and cyclic operad given here is not quite equivalent to that of Getzler--Kapranov: in their definition, graphs (resp.\ forests) are taken to be \emph{stable}, not semistable. That is, modular operads in the sense of Getzler--Kapranov do not have components of genus $0$ and arity $2$, nor of genus $1$ and arity $0$. The reason is that their motivating example is the topological modular operad given by the collection of Deligne--Mumford spaces $\{\Mbar_{g,n}\}$, with operadic structure given by gluing curves along nodes. (A small subtlety is that this is a modular operad not in topological spaces but in orbifolds.) Every modular operad in the sense of Getzler--Kapranov can in a canonical way be considered as a modular operad in the sense defined here, by declaring that the genus $0$ arity $2$ component and the genus $1$ arity $0$ component is the monoidal unit, with obvious structure maps. For this reason, we will not distinguish terminologically between the two notions of modular operad.

Imposing the stability condition $2g-2+n>0$ also allows for examples from hyperbolic geometry: the collection of spaces $\{\M_g^n\}$ of hyperbolic surfaces with parametrized boundaries of length $\varepsilon$ form a modular operad; hyperbolic surfaces can be glued together along their parametrized boundary curves (and here it is important that we insisted that all boundary curves {are totally geodesic and} have exactly the same length). The collection of thick loci in the moduli spaces,
$$ \M_g^n(\varepsilon) \hookrightarrow \M_g^n $$
form a suboperad, and the inclusion of this suboperad is a homotopy equivalence of topological modular operads. Although there is no hyperbolic metric on a cylinder, we may define $\M_0^2$ to be the group $\mathrm{SO}(2)$, with operadic compositions 
$$ \M_0^2 \times \M_g^{n} \to \M_g^n$$
given by reparametrizing the chosen boundary component. We also set $\M_0^2(\varepsilon)=\mathrm{SO}(2)$. With these definitions, both operads $\{\M_g^n\}$ and $\{\M_g^n(\varepsilon)\}$ are equivalent to the surface operad $\{B\Diff(\Sigma_g^n)\}$, except for the genus $1$ arity $0$ component, in which they  differ (but this will not play a role in our paper). Indeed, $B\Diff(\Sigma_0^2) \cong B\Z \cong \mathrm{SO}(2)$, with $\Z$ being generated by a Dehn twist of the cylinder. 

\begin{rem}
    Further variations in the definition of cyclic and modular operad are common in the literature. One can allow as a graph a ``free-floating'' edge with no vertex: since a graph $\DualGraph$ with $n$ legs defines a map $\bigotimes_{v \in V(\DualGraph)} \Oper(n(v)) \to \Oper(n)$ in a cyclic operad $\Oper$, such a free edge defines a map $\mathbf 1 \to \Oper(2)$, making $\Oper$ unital. For our purposes, it will not really matter whether we consider unital or non-unital operads. In addition, many authors (e.g.\ \cite{costelloAinfinityOperad,giansiracusa-handlebodies}) omit the genus decoration on vertices, in which case a grading by genus is an additional structure a modular operad may or may not have; for example, any modular envelope (see the next section) comes with a canonical genus grading.  
\end{rem}

\subsection{The modular envelope and Giansiracusa's theorem}

\newcommand{\Env}{\mathrm{Env}}

Assume now that $\mathcal E$ is cocomplete. Then there is an adjunction $$\mathrm{Cyc} : \mathrm{Fun}(\mathsf{Graphs},\mathcal E) \leftrightarrows \mathrm{Fun}(\mathsf{Forests},\mathcal E) : \Env$$
where $\mathrm{Cyc}$ is the restriction, and $\Env$ the left Kan extension. Although the left Kan extension of a strong monoidal functor is not in general strong monoidal, this is the case here, as explained and placed in a broader context in \cite{bkw}. Thus, the functors $\mathrm{Cyc}$ and $\Env$ restrict to an adjunction between cyclic and modular operads in $\mathcal E$. We call $\Env$ the \emph{modular envelope}.

If $\mathcal O : \mathsf{Forests} \to \mathcal E$ is a cyclic operad, then the colimit formula for Kan extensions gives an explicit formula for $\Env \,\mathcal O$: 
$$ (\Env\, \mathcal O)(g,n) = \colim_{\DualGraph \in F_{g,n}} \mathcal O(\DualGraph).$$
Here $F_{g,n}$ is the slice category $\mathsf{Forests} \downarrow \ast_{g,n}$: its objects are graphs of genus $g$ with $n$ ordered legs, all of whose vertices have genus zero; morphisms are compositions of automorphisms and edge-contractions. We have
$$ \mathcal O(\DualGraph) = \bigotimes_{v \in V(\DualGraph)} \mathcal O(n(v)).$$
The notation is perhaps a bit abusive, and it would be more correct to write $\Oper(\operatorname{Cut} \DualGraph)$. 

If $\mathcal E$ is the category of topological spaces, or chain complexes, or more generally a symmetric monoidal model category, then we can define a derived version of the modular envelope functor, by taking the derived left Kan extension. It admits a similar homotopy colimit formula:
$$ (\mathbf L\Env\, \mathcal O)(g,n) = \hocolim_{\DualGraph \in F_{g,n}} \mathcal O(\DualGraph).$$

A key input into our proof of \cref{mainthm} will be the following theorem of Giansiracusa \cite{giansiracusa-handlebodies}:

\begin{thm}[Giansiracusa] \label{thm: giansiracusa}
    The derived modular envelope of the framed little disks operad is equivalent to the handlebody modular operad, except for the value in genus $1$ and arity $0$.
\end{thm}

One point of the proof will be important for us, namely the passage from semistable graphs to stable graphs \cite[Subsection 6.4]{giansiracusa-handlebodies}. (Giansiracusa's terminology is \emph{reduced} rather than \emph{stable} graphs.) 

Suppose that $\mathcal O$ is a topological cyclic operad $\mathsf{Forests}\to\mathsf{Spaces}$. Then $G=\mathcal O(2)$ is a topological monoid (via operadic composition), and it is equipped with an isomorphism of monoids $\phi : G \stackrel\sim\to G^{\mathrm{op}}$ such that $\phi\circ\phi^{\mathrm{op}} = \mathrm{id}$, via the  $\mathbb S_2$-action on $\mathcal O(2)$. Now let us suppose that $G$ is in fact a \emph{topological group}, in the sense that the monoid structure on $G$ is a group structure, and the isomorphism $G \stackrel\sim\to G^{\mathrm{op}}$ is given by inversion in the group. As explained in the preceding subsection, we have for any topological cyclic operad that
$$ (\mathbf L\Env\, \mathcal O)(g,n) = \hocolim_{\DualGraph \in F_{g,n}} \mathcal O(\DualGraph).$$
When $\mathcal O(2) \cong G$ is a group in the above sense, Giansiracusa explains that we can find a ``smaller'' model of the homotopy colimit on the right hand side:
$$ (\mathbf L\Env\, \mathcal O)(g,n) = \hocolim_{\DualGraph \in F_{g,n}^{\mathrm{stab}}} \mathcal O^{\#}(\DualGraph), \qquad \qquad \text{if } 2g-2+n>0,$$
where $F_{g,n}^{\mathrm{stab}}$ is the subcategory of $F_{g,n}$ consisting only of stable graphs, and 
\[\mathcal O^{\#}(\DualGraph) = \Big(\prod_{v \in V(\DualGraph)} \mathcal O(n(v)) \Big)\Big/ G^{\mathrm{Edge}(\DualGraph)},\]
where each factor $G$ acts on the product via the operadic composition. Explicitly, $G$ acts as multiplication on one of the two factors, and by composing with the inverse along the other factor, and the quotient does not depend on which of the two factors is chosen. The quotient is taken to be a homotopy quotient. 

Informally, the reason for this is that allowing graphs to have bivalent vertices means that along each edge of a stable graph, we allow arbitrary subdivisions, and all of these subdivisions together make up a copy of a Borel construction for the action of $G$. Taking the homotopy colimit therefore gives a model for the homotopy quotient. An equivalent point of view is that  $\mathcal O^\#$ is the derived left Kan extension of $\mathcal O$ along the stabilization functor $F_{g,n}\to F_{g,n}^{\mathrm{stab}}$; see \cite[Lemma 6.4.1]{giansiracusa-handlebodies}.

\subsection{Cofibrant diagrams and manifolds-with-corners}
\label{subsec: cofibrant diagrams}

Define $G_{g,n}$ to be the slice category $\mathsf{Graphs} \downarrow \ast_{g,n}$ of all semistable graphs of genus $g$ with $n$ legs, and define $G_{g,n}^\mathrm{stab}$ as the subcategory of stable graphs. 

The categories $G_{g,n}^{\mathrm{stab}}$ and its subcategory $F_{g,n}^{\mathrm{stab}}$ are naturally \emph{generalized Reedy categories} \cite{bergermoerdijk}. We will choose as our degree function $d:G_{g,n}^{\mathrm{stab}} \to \N$ the function $$d(\DualGraph) = \sum_{v \in V(\DualGraph)} 3g(v)-3+n(v);$$
the choice is motivated by the fact that this is the complex dimension of the stratum inside $\Mbar_{g,n}$ corresponding to $\DualGraph$. 

Let $\mathcal O$ denote the topological modular operad $\{\M_g^n(\varepsilon)\}$, and let $G=\mathcal O(0,2)=\mathrm{SO}(2)$. We may define a functor $\mathcal O^\#$ from $G_{g,n}^{\mathrm{stab}}$ to spaces, by repeating the construction explained in the preceding subsection:
\[\mathcal O^{\#}(\DualGraph) = \Big(\prod_{v \in V(\DualGraph)} \mathcal O(g(v),n(v)) \Big)\Big/ G^{\mathrm{Edge}(\DualGraph)}.\]
If $n>0$, then $G^{\mathrm{Edge}(\DualGraph)}$ acts freely on $\prod_{v \in V(\DualGraph)} \mathcal O(g(v),n(v))$ for all $\DualGraph \in G_{g,n}^{\mathrm{stab}}$, and we can substitute the homotopy quotient in the definition of $\mathcal O^\#$ with an actual quotient. This is because the space $\M_g^n$, which a priori is an orbifold, is an actual manifold for $n>0$. If $n=0$, then $G^{\mathrm{Edge}(\DualGraph)}$ acts with finite stabilizers, and we will in this case consider the functor $\mathcal O^\#$ as being valued in differentiable or topological stacks.

\begin{lem} \label{lem:cofibrant} Let $n>0$. The functor $\mathcal O^\# : G_{g,n}^{\mathrm{stab}} \to \mathsf{Spaces}$, with notation as in the preceding subsection, is generalized Reedy cofibrant. 
\end{lem}


\begin{proof}
     Since every non-isomorphism in $G_{g,n}^{\mathrm{stab}}$ raises 
 degree, $(G_{g,n}^{\mathrm{stab}})^-$ contains only the isomorphisms and $(G_{g,n}^{\mathrm{stab}})^+=G_{g,n}^{\mathrm{stab}}$.
    Now for every graph $\DualGraph$, the latching object $L_\DualGraph \Oper^{\#}$ is the colimit
    \begin{equation*}
        L_\DualGraph \Oper^{\#} = \colim_{\DualGraph'\to \DualGraph}{\Oper^{\#}(\DualGraph')}
    \end{equation*}
    with the colimit ranging over all non-invertible morphisms $\DualGraph'\to \DualGraph$.

    We need to prove that for every stable graph $\DualGraph$ the map $L_\DualGraph \Oper^{\#} \to \Oper^{\#}(\DualGraph)$ is an $\operatorname{Aut}(\DualGraph)$-equivariant cofibration.

    This is clearest when $g=0$, in which case $G_{0,n}^{\mathrm{stab}}$ is a poset. In fact, $\mathcal O^\#(\DualGraph)$ is a manifold with corners for all $\DualGraph$ of genus zero, the poset $G_{0,n}^{\mathrm{stab}}$ is the poset of strata of the manifold $\mathcal O^\#(\ast_{0,n})$, and the diagram $\mathcal O^\# : G_{0,n}^{\mathrm{stab}} \to \mathsf{Spaces}$ is given by taking the closures of strata. This is clearly Reedy cofibrant, and the latching object is simply the union of all strata of positive codimension.  

    An analogous statement is true also when $g>0$. However, one must deal with the fact that graphs can have automorphisms (and hence we have a generalized Reedy category). Let $G_{g,n}^\mathrm{stab}/\!\simeq$ be the set of isomorphism classes of the category $G_{g,n}^\mathrm{stab}$. It naturally forms a poset. The space $\mathcal O^\#(\ast_{g,n})$ is a manifold-with-corners, and its poset of strata is $G_{g,n}^\mathrm{stab}/\!\simeq$. The open stratum corresponding to a graph $\DualGraph$ is the quotient
    $$ \big(\operatorname{Int} \mathcal O^\#(\DualGraph)\big) \big/ \operatorname{Aut}(\DualGraph),$$
    where by $\operatorname{Int}(-)$ we mean the interior of a manifold-with-corners. It follows that the finite group $\operatorname{Aut}(\DualGraph)$ acts freely on both $L_\DualGraph \Oper^{\#}$ and $\Oper^{\#}(\DualGraph)$, for all $\DualGraph$, and that the induced map 
    $$ L_\DualGraph \Oper^{\#}/\mathrm{Aut}(\DualGraph) \longrightarrow \Oper^{\#}(\DualGraph)/\mathrm{Aut}(\DualGraph)$$
    is the inclusion of the union of all strata of positive codimension, and in particular a cofibration. Hence $L_\DualGraph \Oper^{\#} \to \Oper^{\#}(\DualGraph)$ is an equivariant cofibration.
    %
    %
%
\end{proof}
When $n=0$ the analogous statement cannot literally be correct, since $\mathcal O^\#(g,0)$ is an orbifold and there is to our knowledge no appropriate model structure on topological or differentiable stacks --- they do not even form a $1$-category. However, it is the case that $\mathcal O^\#(\ast_{g,0}) = \M_g(\varepsilon)$ is an orbifold-with-corners, and its poset of strata is $G_{g,0}^\mathrm{stab}/\!\simeq$. The open stratum corresponding to a graph $\DualGraph$ is the quotient
    $$ \big(\operatorname{Int} \mathcal O^\#(\DualGraph)\big) \big/ \operatorname{Aut}(\DualGraph).$$
    If we now fix a finite cover $X\to \M_g$ by a complex manifold, e.g.\ by putting a level structure on the Riemann surfaces parametrized by $\M_g$, then we may consider the inverse image of $\M_g(\varepsilon)$ inside $X$. Now it becomes a manifold with corners, and the diagram of its strata is cofibrant as before.

\subsection{Hyperbolic model for handlebodies}
Recall that a \emph{short geodesic} on a hyperbolic surface is a closed geodesic $\geodesic$ with $\ell(\geodesic) < \varepsilon$. Let $2g-2+n+m>0$. Recall from the introduction that $\HM_{g,n}^m \subset \M_{g,n}^m$ denotes the open locus parametrizing hyperbolic surfaces with the property that the collection of all short geodesics cut it into pieces of genus zero. We define $\overline{\mathcal H}\M_{g,n}^m$ to be the closure of $\HM_{g,n}^m$ inside $\M_{g,n}^m$. Thus, $\overline{\mathcal H}\M_{g,n}^m$ parametrizes curves which are cut into genus zero pieces by the collection of geodesics with $\ell(\geodesic) \leq \varepsilon$. Note that $\HM_{0,n}^m = \M_{0,n}^m$. Finally, we define 
$$ \mathcal X_{g,n}^m := \overline{\mathcal H}\M_{g,n}^m \cap \M_{g,n}^m(\varepsilon).$$
Explicitly, $\mathcal X_{g,n}^m$ parametrizes hyperbolic surfaces with all closed geodesics having length at least $\varepsilon$, such that those of length exactly $\varepsilon$ cut it into pieces of genus zero.

\begin{prop}\label{prop: two equivalences}
    The maps
    \[ \HM_{g,n}^m \hookrightarrow \overline{\mathcal H}\M_{g,n}^m \hookleftarrow \mathcal X_{g,n}^m\]
    are homotopy equivalences. 
\end{prop}

\begin{proof}It follows from the discussion in Subsection \ref{subsec: thick} that $\overline{\mathcal H}\M_{g,n}^m$ is a manifold-with-corners (orbifold, when $m=0$) and that $\HM_{g,n}^m$ is its interior. Thus, the first claim is clear. 

For the second claim, we use the deformation retraction of $\M_{g,n}^m$ down to $\M_{g,n}^m(\varepsilon)$ constructed by Ji--Wolpert, described in the end of Subsection \ref{subsec: thick}. It is easy to see from the description of the deformation retraction given there that it restricts to a deformation retraction of $\overline{\mathcal H}\M_{g,n}^m$ onto $\mathcal X_{g,n}^m$.\end{proof}

The modular operad structure on the spaces $\{\M_g^m(\varepsilon)\}$ restricts to a modular operad structure on the spaces $\{\mathcal X_g^m\}$, by gluing hyperbolic surfaces along parametrized boundaries. 

\begin{thm}\label{thm: mainthm}
The modular operad given by the spaces $\{\mathcal X_g^m\}$, where $2g-2+m>0$, is equivalent to the handlebody modular operad.
\end{thm}

\begin{proof}As in Subsection \ref{subsec: cofibrant diagrams} we let $\Oper$ denote the modular operad of thick moduli spaces $\{\M_g^m(\varepsilon)\}$, with $\Oper(0,2)=\mathrm{SO}(2)$. We denote by $\Oper_c = \operatorname{Cyc}\Oper$ its restriction to a cyclic operad, which is equivalent to the framed little disks operad. Giansiracusa's \cref{thm: giansiracusa} shows that the derived modular envelope $\mathbf L \Env\,\Oper_c$ is equivalent to the handlebody modular operad, except in genus $1$ and arity $0$. 

As explained after \cref{thm: giansiracusa}, $\mathbf L \Env\,\Oper_c (g,n)$ is the homotopy colimit of $\Oper_c^\#$ over the category $F_{g,n}^{\mathrm{stab}}$. But according to \cref{lem:cofibrant}, the functor $\Oper^\#_c$ is a cofibrant diagram of spaces over $G_{g,n}^\mathrm{stab}$, and hence also over its subcategory $F_{g,n}^\mathrm{stab}$, so we can replace the homotopy colimit with an ordinary colimit. When $n=0$ this does not follow from \cref{lem:cofibrant}, but from the discussion immediately following the lemma. In any case, the ordinary colimit is nothing but the modular suboperad of $\Oper$ given by the image of the counit of the adjunction
$$ \Env\operatorname{Cyc}\Oper \to \Oper,$$
and unwinding definitions shows that the image of the counit is nothing but the topological modular operad $\{\mathcal X_g^m\}$. 
\end{proof}

\begin{cor}
The space $\HM_{g,n}^m$ is a classifying space for the handlebody group $\HMod_{g,n}^m$. 
\end{cor}

\begin{proof}When $n=0$ this follows from \cref{thm: mainthm} and \cref{prop: two equivalences}. The general case also follows from this. Indeed, we consider firstly the short exact sequence
\begin{equation}
    0 \to \Z^{n} \to \HMod_{g}^{n+m} \to \HMod_{g,n}^m \to 1 \label{eq: seq}
\end{equation} 
    where $\Z^n$ denotes the subgroup generated by Dehn twists around $n$ of the $(n+m)$ boundary components. We also know geometrically that $\M_{g}^{n+m}$ is the total space of an $n$-fold circle bundle over $\M_{g,n}^m$, by identifying $\M_{g,n}^m$ with the quotient of $\M_g^{n+m}$ by the action of $\mathrm{SO}(2)^n$ given by reparametrizing $n$ of the boundaries. By restriction, $\HM_{g}^{n+m}$ becomes an $n$-fold circle bundle over $\HM_{g,n}^m$, and we may consider the induced long exact sequence on homotopy groups. Comparing this long exact sequence with \eqref{eq: seq}, we see that it suffices to show that $\HM_g^{n+m} \simeq B\HMod_g^{n+m}$ to deduce that $\HM_{g,n}^{m} \simeq B\HMod_{g,n}^{m}$.
\end{proof}

\section{The tropicalization map}

In Subsection \ref{subsec: moduli of tropical curves} we made use of the \emph{tropicalization} morphism $\lambda_{g,n} : \M_{g,n}\to\M_{g,n}^\trop$. This map was introduced by Chan--Galatius--Payne \cite[\S 7.1]{changalatiuspayne}, as an analogue of a similar non-archimedean tropicalization map $\M_{g,n}^{\mathrm{Berk}}\to\M_{g,n}^\trop$, see Abramovich--Caporaso--Payne \cite{acp}. The former map is defined on the \emph{complex} analytification of the algebraic stack $\M_{g,n}$, and the latter on the \emph{Berkovich} analytification.

The goal of this section of the paper is to prove that the map $\lambda_{g,n}$ is continuous and proper, and that it induces an isomorphism $H^\bullet_c(\M_{g,n}^\trop,\Q) \stackrel \sim \longrightarrow W_0 H^\bullet_c(\M_{g,n},\Q).$ These properties were stated without proof in \cite[\S 7.1]{changalatiuspayne}.

At this point we also need to make an important disclaimer. Thus far, we have considered the space $\M_{g,n}$ as a \emph{stack} throughout, without much fuss. It is perhaps most naturally an algebraic or analytic stack, in which case it represents a geometrically natural functor \cite{hinich-vaintrob}, but can also be considered by analytification as a \emph{topological} stack, or an orbifold. Similarly, after \cite{tropicalstack} the space $\M_{g,n}^\trop$ is most naturally considered as a stack on a certain category of polyhedral cones, where it represents a moduli functor of tropical curves. By analytification it may be considered as a topological stack, too. Now one certainly expects there to be a well-behaved tropicalization morphism of \emph{topological stacks} $\M_{g,n}\to\M_{g,n}^\trop$ --- however, we do not actually know how to prove this, the basic problem being that we do not even know what functor either $\M_{g,n}$ or $\M_{g,n}^\trop$ represents on the site of topological spaces. 

In this section, we therefore construct the tropicalization map only as a morphism between the \emph{coarse moduli spaces} of the respective topological stack: to do this, it suffices to define set-theoretically a function that assigns to an isomorphism class of hyperbolic surfaces an isomorphism class of tropical curves, and then to verify in local coordinates that the function is continuous. For our intended applications, knowing that the tropicalization morphism exists as a map of coarse spaces is completely sufficient, since the projection to the coarse space induces an isomorphism in rational cohomology (and rational cohomology with compact support). 

In what follows, we write $M_{g,n}$ (resp.\ $M_{g,n}^\trop$) for the coarse spaces of $\M_{g,n}$ and $\M_{g,n}^\trop$. 

In fact, for our arguments, it will be convenient to consider an extension of the tropicalization morphism to the Deligne--Mumford boundary, giving a continuous map $\overline M_{g,n} \to \overline M_{g,n}^\trop$. The space $\Mbar_{g,n}^\trop$ is the compactification of $\M_{g,n}^\trop$ obtained by allowing $\infty$ as an edge length of the metric graphs \cite{cap}, and $\overline M_{g,n}^\trop$ is its coarse space. 

Consider a point $[X] \in \overline{M}_{g,n}$. We consider $X$ as a nodal hyperbolic surface, with a cusp at each puncture and node. We define $\lambda_{g,n}[X]$ to be the following metric graph. The underlying graph $\DualGraph$ is the dual graph of the surface obtained from $X$ by collapsing every short geodesic on $X$ to a point. By the collar lemma, this is well-defined. The metric on $\DualGraph$ is defined as follows: the edges corresponding to nodes of $X$ are given the edge length $\infty$. An edge corresponding to a short geodesic on $X$ of length $\ell$ is given the edge length $-\log(\ell/\epsilon)$. 

\begin{prop}
	The function $\lambda_{g,n} : \overline{M}_{g,n} \to \overline M_{g,n}^\trop$ just defined is continuous.
\end{prop}

\begin{proof}
	It suffices to check this locally in the coordinates provided by the Bers atlas, as the projection $\Mbar_{g,n}\to\overline M_{g,n}$ is locally a quotient map. So pick a point $[X] \in \overline M_{g,n}$. To define a Bers chart around $[X]$ we fix a surjection $S \to X$ from the chosen reference surface of genus $g$ with $n$ punctures. Take a pants decomposition of $S$ which contains all the curves that are contracted to nodes in $X$, and in addition all curves whose image in $S$ is a closed geodesic shorter than $\log(3+2\sqrt 2)$. Indeed, by the collar lemma these curves are disjoint and therefore extend to a pants decomposition.
	
	We claim now that there is an open neighborhood of $[X]$ such that for any hyperbolic metric defined by a point in this neighborhood, any short geodesic (or curve contracted to a node) is included in the fixed pants decomposition of $S$. The conclusion will follow from this claim, since on this neighborhood, with respect to the local coordinates with respect to the Bers atlas of our chosen pants decomposition, the map to $\overline M_{g,n}^\trop$ is manifestly continuous.
	
	If $[X]$ lies in the interior of $M_{g,n}$, then the claim is an immediate consequence of Wolpert's lemma, and that we chose $\varepsilon$ to be \emph{strictly} smaller than $\log(3+2\sqrt 2)$. If $[X]$ lies on the Deligne--Mumford boundary, then we may consider its normalization to obtain a point in a product of smaller Teichm\"uller spaces, to which Wolpert's lemma can be applied. But we also need to consider the set of curves $\mathfrak C = \{\geodesic \subset S : \ell_\geodesic[X] = \infty\}$, i.e.\ those curves on $S$ whose image on $X$ passes through a node. We claim that for any $M \in \R$, the set $\{[Y] \in \C^{3g-3+n} : \ell_\geodesic[Y]>M \text{ for all } \geodesic \in \mathfrak C\}$ is an open neighborhood of $X$, which will finish the proof. Indeed, Wolpert's lemma and the discreteness of the length spectrum of a Riemann surface \cite[Lemma 12.4]{farbmargalit} implies that the infinite intersection $\bigcap_{\geodesic \in \mathfrak C} \{[Y]: \ell_\geodesic[Y]>M\}$ is locally defined by finitely many hypersurfaces, and therefore is open. 
\end{proof} 

By restriction, we also obtain a continuous function $M_{g,n}\to M_{g,n}^\trop$. We record the following observation:

\begin{cor}
	The map $M_{g,n}\to M_{g,n}^\trop$ is proper.
\end{cor}

\begin{proof}
	Indeed, it extends to a continuous map between compactifications. 
\end{proof}

\begin{prop}\label{thm tropicalization} The map $H^\bullet_c(M_{g,n}^\trop,\Q) \to H^\bullet_c(M_{g,n},\Q)$ is injective, and its image is the weight zero part of the mixed Hodge structure on $H^\bullet_c(M_{g,n},\Q)$. 
	
\end{prop}

\begin{proof}
	Consider the stratification of $\Mbar_{g,n}$ by topological type, and the induced stratification of $\overline M_{g,n}$. There is a corresponding stratification of $\Mbar_{g,n}^\trop$, where we stratify according to which edges have length $\infty$, which similarly induces a stratification of $\overline M_{g,n}^\trop$. The stratifications are compatible, in the sense that the stratification of $\overline M_{g,n}$ by topological type is the pullback of the stratification of  $\overline M_{g,n}^\trop$ along the map $\lambda_{g,n}$. The open dense strata in the respective stratifications are $M_{g,n}$ and $M_{g,n}^\trop$, respectively. 
	
	By the construction of \cite{spectralsequencestratification} there is in this situation associated two spectral sequences converging to $H^\bullet_c(M_{g,n})$ and $H^\bullet_c(M_{g,n}^\trop)$, respectively. By functoriality of the construction of \cite{spectralsequencestratification} there is a morphism between the two spectral sequences, which on abutments gives the map $H^\bullet_c(M_{g,n}^\trop,\Q) \to H^\bullet_c(M_{g,n},\Q)$. 
	
	When the general construction of \cite{spectralsequencestratification} is applied to a smooth projective variety stratified by a normal crossing divisor, one obtains the Poincar\'e dual of the spectral sequence used by Deligne to define the mixed Hodge structure on a smooth variety in \cite{hodge2}, see \cite[Example 3.5]{spectralsequencestratification}. This applies in particular to the stratification of $\Mbar_{g,n}$ by topological type. On the $E_1$-page we see the compact support cohomology of closures of strata. These closures are products of smaller spaces $\Mbar_{g',n'}$ (possibly up to taking the quotient by an action of a finite group). In particular, the weight zero part is concentrated along a single row of the spectral sequence, given by $H^0$ of the closures of strata, as the strata closures are smooth and proper. From \cite{hodge2}, and/or compatibility of the differentials in the spectral sequence with weights, the spectral sequence degenerates after the $E_1$-differential. 
	
	Considering instead $\Mbar_{g,n}^\trop$, we instead see that the closures of strata are isomorphic to products of smaller spaces $\Mbar_{g',n'}^\trop$, and hence are compact and contractible. In particular, they have the same compact support cohomology as a point. The $E_1$-page of the spectral sequence is now concentrated along a single row. In fact, the resulting chain complex is nothing but the commutative graph complex with loops and weights, since strata are indexed by dual graphs, with adjacencies defined by edge contractions. In any case, the map on $E_1$-pages between the two spectral sequences is simply given by the inclusion of the row which defines the weight zero part of the cohomology. 
 
 Finally, the projection $\Mbar_{g,n}\to\overline M_{g,n}$ induces an isomorphism between the spectral sequences of the respective stratifications, and similarly for $\Mbar_{g,n}^\trop \to\overline M_{g,n}^\trop $. The result follows. \end{proof}

\begin{rem}The proof of the preceding proposition was somewhat dishonest: it is not correct that the closures of strata inside $\overline M_{g,n}$ are products of smaller moduli spaces, as stated in the proof. The problem is that a gluing map, say $\overline{M}_{g,n+1} \times \overline M_{g',n'+1} \to \overline M_{g+g',n+n'}$, is typically injective on the interior of the moduli space, but many-to-one on the boundary. What is described in the proof is not the actual closures of strata (which would be a correct description if we were on a normal crossing divisor on an actual smooth variety) but rather the components of the normalization of the locus of boundary points of multiplicity $\geq k$, for some $k$. Fortunately, the latter is what actually appears in Deligne's spectral sequence. The fact that the constructions involving Deligne's spectral sequence, and the spectral sequence of \cite{spectralsequencestratification}, work in the orbifold setting, relies on the fact that both constructions are local and sheaf-theoretic in nature and can be carried out on an \'etale neighborhood of a point which is an actual complex manifold. 
\end{rem}

\begin{rem}
	The family of cohomology groups $\{H^\bullet(\Mbar_{g,n},\Q)\}$ form the \emph{cohomological field theory cooperad} $\mathsf{CohFT}$. (The cyclic part of this modular cooperad is often called the \emph{hyper\-commutative} cooperad.) The degree zero part, $H^0(\Mbar_{g,n},\Q) \hookrightarrow H^\bullet(\Mbar_{g,n},\Q)$, forms a sub-cooperad, which we denote $\mathsf{TFT}$. Thus $\mathsf{TFT}$ is the (underived) modular envelope of the commutative cooperad; it is given by the trivial representation of $\mathbb S_n$ concentrated in degree $0$ in each arity and genus. 
	
	The Feynman transform of $\mathsf{CohFT}$ is equivalent to the gravity operad, as follows by combining \cite[Proposition 6.11]{getzlerkapranov} and the formality of the operad $\{\Mbar_{g,n}\}$, which in turn follows from \cite{dgms,formaloperads}, see also \cite{cirici-horel}. Explicitly, by the gravity operad we mean here the $\mathfrak K$-twisted dg modular operad given by compactly supported cochains $\{C_c^\bullet(\M_{g,n})\}$, with operadic structure maps given by the connecting homomorphisms arising from the embedding of $\M_{g,n+1} \times \M_{g',n'+1}$ as a codimension $1$ boundary stratum in the compactification of $\M_{g+g',n+n'}$. 
	
	The Feynman transform of $\mathsf{TFT}$ is by definition the commutative graph complex with loops and weights, which as explained in \cite{changalatiuspayne} computes $H^\bullet_c(\M_{g,n}^\trop,\Q)$, and also the weight zero part of $H^\bullet_c(\M_{g,n},\Q)$. 
	
	It follows from the proof of \cref{thm tropicalization} that the geometrically defined homomorphism $\lambda_{g,n}^\ast : H^\bullet_c(M_{g,n}^\trop,\Q)\to H^\bullet_c(M_{g,n},\Q)$ coincides with the map obtained by applying the Feynman transform to the inclusion $\mathsf{TFT} \to \mathsf{CohFT}$. Indeed, the projection $\overline M_{g,n} \to \overline M_{g,n}^\trop$ induces a map on cohomology equivalent to $H^0(\overline M_{g,n},\Q) \hookrightarrow H^\bullet(\overline M_{g,n},\Q)$. Moreover, the identification of the Feynman transform of $\mathsf{CohFT}$ with the gravity operad proceeds precisely via Deligne's spectral sequence, and the identification of the Feynman transform of $\mathsf{TFT}$ with the $\mathfrak K$-twisted  modular operad of compactly supported cochains on $M_{g,n}^\trop$ --- what might be called the ``tropical gravity operad'' --- can be understood similarly. 
\end{rem}

\section{Two Leray--Serre spectral sequences}\label{section Leray Serre}

The goal of this section is to discuss and motivate the conjecture stated in \cref{rem: conjecture}, and to indicate why handlebodies appear naturally when thinking about this conjecture. To simplify the notation and discussion we suppose $g \geq 2$. 

\begin{conj}\label{conjecture}There is a natural isomorphism between the two spectral sequences 
\[\begin{tikzcd}
   E_2^{pq} = H^p_c(\CV_g,R^q \rho_! \Q) \arrow[r,Rightarrow] \arrow[d,dashed,"\cong"] &  H^{p+q}_c(\CV_{g,n},\Q) \arrow[d,"\cong"]\\
   E_2^{pq} = W_0 H^p_c(\M_g,R^q \pi_! \Q)  \arrow[r,Rightarrow] & W_0 H^{p+q}_c(\M_{g,n},\Q)
\end{tikzcd}\]
lifting the isomorphism between abutments, where $\rho : \CV_{g,n} \to \CV_g$ and $\pi:\M_{g,n}\to\M_g$ are the respective forgetful maps. 
\end{conj}

 The first of these spectral sequences is the subject of forthcoming work of Bibby--Chan--Gadish--Yun \cite{bibby}. They show in particular that the local system $R^q \rho_! \Q$ is  isomorphic to $H^q_c(F(\vee_g S^1,n),\Q)$ (the compactly supported cohomology of the configuration space of $n$ distinct ordered points on a wedge of $g$ circles), with its natural action of $\mathrm{Out}(F_g) \simeq \mathrm{hAut}(\vee_{g} S^1) $, as studied in \cite{gadish,gadish-hainaut}. This is a somewhat subtle fact --- indeed, it requires treating the spaces $\CV_{g,n}$ as topological stacks throughout, and moreover, the fibers of $\rho$ are certainly not configuration spaces of points on the tropical curves. In any case, the local system $R^q \pi_! \Q$ is of quite different nature, being given by $H^q_c(F(\Sigma_g,n),\Q)$, where $\Sigma_g$ is an oriented closed genus $g$ surface, with its natural action of $\Mod_g$. Moreover, to make sense of the lowest weight subspace $W_0 H^p_c(\M_g,R^q \pi_! \Q)$ we must treat $R^q \pi_! \Q$ as an admissible variation of mixed Hodge structure, so that the work of Saito \cite{saitomixedhodge} furnishes a mixed Hodge structure on its cohomology.

On the face of it, there is not even a natural map $H^p_c(\CV_g,R^q \rho_! \Q) \to H^p_c(\M_g,R^q\pi_!\Q)$ that would be a \emph{candidate} for providing the isomorphism in \cref{conjecture} But the space $\HM_g$, and the interpretation in terms of handlebodies, furnishes such a map. Indeed, consider the tropicalization map $\lambda : \HM_g \to \CV_g$, and take the pullback $\lambda^\ast R^q \rho_!\Q$. Identifying $\HM_g$ with the moduli space of handlebodies, we can now identify this local system with $H^q_c(F(V_g,n),\Q)$, equipped with its natural action of $\HMod_g$. We can also restrict $R^q \pi_!\Q$ to $\HM_g$. The closed embedding $F(\Sigma_g,n) \hookrightarrow F(V_g,n)$ defines a map of local systems on $\HM_g$,
$$ \lambda^\ast R^q \rho_!\Q \longrightarrow R^q \pi_!\Q. $$
We can now formulate a more precise form of \cref{conjecture}: the composite $$ H^p_c(\CV_g,R^q\rho_!\Q) \longrightarrow H^p_c(\HM_g,\lambda^\ast R^q \rho_!\Q) \longrightarrow H^p_c(\HM_g,R^q \pi_!\Q) \longrightarrow H^p_c(\M_g,R^q \pi_!\Q)$$
induces an isomorphism $H^p_c(\CV_g,R^q\rho_!\Q) \cong W_0H^p_c(\M_g,R^q \pi_!\Q) $. Note that this composite does indeed define a map of spectral sequences from the compactly supported Leray--Serre spectral sequence for $\rho$, to the compactly supported Leray--Serre spectral sequence for $\pi$.

What gives \cref{conjecture} weight is that the local systems $R^q\rho_!\Q$ and $R^q \pi_!\Q$ are more closely related than one may at first expect; let us try to justify this claim. It is easiest to first analyze the semisimplifications of the two local systems. On the semisimplification of $R^q \rho_! \Q$, the monodromy factors through the surjection $\mathrm{Out}(F_g) \to \mathrm{GL}(g,\Z)$, and up to semisimplification one has $R^q \rho_! \Q \cong \bigoplus_i V_{\mu_i}$, where $V_\mu$ denotes an irreducible algebraic representation of $\mathrm{GL}(g)$ with highest weight vector $\mu = (\mu_1 \geq \ldots \geq \mu_g \geq 0)$. Similarly, the monodromy of the semisimplification of $R^q \pi_!\Q$ factors through $\Mod_g \to \mathrm{Sp}(2g,\Z)$, and up to semisimplification one has $R^q \pi_!\Q \cong \bigoplus_j \mathbb V_{\lambda_j}(-n_j)$; here $\mathbb V_\lambda$ denotes the polarized variation of Hodge structure on $\M_g$ of weight $\vert\lambda\vert$ associated to the representation of highest weight $\lambda = (\lambda_1 \geq \ldots \geq \lambda_g \geq 0)$, and the integer $n_j \geq 0$ indicates a Tate twist. 

The paper \cite{configuration} defines a chain complex $\mathcal{CF}(A,n)$ associated to a cdga $A$ and a positive integer $n$ (this chain complex was denoted $\mathcal{CF}(\Pi_n\!\setminus\!\{\hat{0}\},A)$ in loc. cit.). The complex $\mathcal{CF}(A,n)$ is an explicit direct sum of degree-shifted tensor powers of $A$, with a differential given by the cup-product and internal differential on $A$. When $A$ is a cdga model for $C^\bullet_c(X,\Q)$, then this chain complex computes $H^\bullet_c(F(X,n),\Q)$. If $X$ is compact and formal, then $A$ may be taken to be the cohomology ring $H^\bullet(X,\Q)$. If $X$ is algebraic, and one keeps track of the mixed Hodge structure on $H^\bullet(X,\Q)$, then the construction recovers the mixed Hodge structure on $H^\bullet_c(F(X,n),\Q)$ up to semisimplification.

Let us make a brief purely representation-theoretic interlude. Consider a $2g$-dimensional symplectic vector space $\mathbb V$ and a $g$-dimensional \emph{Lagrangian quotient} $\mathbb V \to L$. Schur--Weyl duality, and Weyl's construction of the irreducible representations of the symplectic group \cite[\S17.3]{fultonharris}, tells us how to decompose the tensor powers of $\mathbb V$ and $L$ into irreducible representations of $\mathrm{Sp}(2g)$ and $\mathrm{GL}(g)$, respectively. One finds that 
$$ L^{\otimes k} \cong \bigoplus_{|\lambda | = k} S^\lambda(L) \otimes \sigma_\lambda,$$
where $S^\lambda(-)$ denotes a Schur functor and $\sigma_\lambda$ a Specht module,
and 
$$ \mathbb V^{\otimes k} \cong \bigoplus_{\substack{|\lambda| \leq k \\ |\lambda| \equiv k \pmod 2}} \mathbb V_\lambda \otimes \beta_{\lambda,k},$$
where $\mathbb V_\lambda$ is the irreducible representation of $\mathrm{Sp}(2g)$ of highest weight $\lambda$, and $\beta_{\lambda,k}$ is a certain module over the Brauer algebra on $k$ strands. When $\vert\lambda\vert=k$, one has $\beta_{\lambda,k} \cong \sigma_\lambda$. Moreover, one sees precisely how these summands behave under the natural projection $\mathbb V^{\otimes k} \to L^{\otimes k}$: each summand $\mathbb V_\lambda$ with $\vert\lambda\vert < k$ is mapped to $0$ (since $L$ was Lagrangian), and when $\vert\lambda\vert=k$ the summand $\mathbb V_\lambda \otimes \sigma_\lambda$ surjects naturally to $S^\lambda(L) \otimes \sigma_\lambda$. 

Returning to the topic at hand, when $X=\Sigma_g$ is a smooth projective curve of genus $g$, one has $$ A_\Sigma := H^\bullet(\Sigma_g,\Q) = \Q(0)[0] \oplus \mathbb V[-1] \oplus \Q(-1)[-2],$$
where $\mathbb V$ denotes the first cohomology group of the curve, which we identify with the defining representation of $\mathrm{Sp}(2g,\Z)$. If we choose a thickening to a handlebody $\Sigma_g \hookrightarrow V_g$, then 
$$ A_V := H^\bullet(V_g,\Q) = \Q[0] \oplus V[-1],$$
where we identify $V$ with a Lagrangian \emph{subspace} of $\mathbb V$. The inclusion $ A_V \to A_\Sigma$ induces a chain map $\mathcal{CF}(A_V,n)\to \mathcal{CF}(A_\Sigma,n)$, which on homology gives rise to the natural map $H^\bullet_c(F(V_g,n),\Q) \to H^\bullet_c(F(\Sigma_g,n),\Q)$ featuring in the conjecture.

Now observe that if we are only interested in the weight zero part of $H^p_c(\M_g,R^q \pi_! \Q)$, then all summands of $R^q \pi_! \Q$ involving a positive Tate twist (i.e.\ $n_j>0$) can be ignored. We claim that discarding terms with $n_j>0$ exactly corresponds to thickening $\Sigma_g$ to a handlebody. Indeed, all summands in $\mathcal{CF}(A_\Sigma,n)$  involving $H^2(\Sigma_g,\Q) \cong \Q(-1)$ will give rise to a strictly positive Tate twist. Moreover, recall that $\mathcal{CF}(A,n)$ was defined as a sum of tensor powers of $A$, so that we are led to decomposing the tensor powers of $A_\Sigma$. Since $\mathbb V^{\otimes k}$ is pure of weight $k$, it follows that when it is decomposed into irreducible representations $\mathbb V_\lambda$, then those $\lambda$ with $\vert\lambda\vert <k$ are precisely the summands with a non-zero Tate twist. From our representation-theoretic interlude --- or rather its linear dual, since we have a Lagrangian subspace rather than a Lagrangian quotient --- it follows that the inclusion $\mathcal{CF}(A_V,n)\to \mathcal{CF}(A_\Sigma,n)$ hits precisely those summands that occur with vanishing Tate twist, and the subspace of $\mathcal{CF}(A_\Sigma,n)$ with no Tate twist can be obtained from $\mathcal{CF}(A_V,n)$ by formally replacing each summand $S^\lambda(V)$ with the irreducible representation $\mathbb V_\lambda$. 

We may summarize the above discussion in the following theorem: 
\begin{thm}\label{multiplicity thm}
For any partition $\lambda$, the multiplicity of the representation $V_\lambda$ of $\mathrm{GL(g)}$ in the semisimplification of $H^q_c(F(\vee_g S^1,n),\Q)$ always coincides with the multiplicity of the representation $\mathbb V_\lambda$ of $\mathrm{Sp}(2g)$ inside $\mathrm{Gr}^W_{\vert\lambda\vert} H^q_c(F(\Sigma_g,n),\Q)$. In fact, the two multiplicity spaces coincide as representations of $\mathbb S_n$.
\end{thm} Note that applying $\mathrm{Gr}^W_{\vert\lambda\vert}$ ensures that $\mathbb V_\lambda$ occurs without Tate twist, and that the representation is  semisimple. Based on \cref{multiplicity thm} and \cref{conjecture}, it is natural to conjecture moreover the following:
\begin{conj}\label{conjecture 2}For any partition $\lambda = (\lambda_1 \geq \ldots \geq \lambda_g \geq 0)$, one has
    $$ H^\bullet_c(\CV_g,V_\lambda) \cong W_0 H^\bullet_c(\M_g,\mathbb V_\lambda). $$
\end{conj}
Here $\mathbb V_\lambda$ is the $\Q$-PVHS of weight $\vert\lambda\vert$ whose underlying local system is given by the irreducible representation of highest weight $\lambda$, whereas $V_\lambda$ is the irreducible representation of $\mathrm{GL}(g)$ of highest weight $\lambda$. 

\cref{conjecture 2} is known when $g=2$. One can deduce this from the $g=2$ case of \cref{multiplicity thm}, given that both spectral sequences in \cref{conjecture} degenerate immediately and that the sheaves $R^q \rho_!\Q$ and $R^q \pi_!\Q$ are semisimple  when $g=2$. For the map $\rho$, semisimplicity and degeneration follows because the Torelli subgroup of $\mathrm{Out}(F_2)$ is trivial, i.e.~$\mathrm{Out}(F_2)\to\mathrm{GL}(2,\Z)$ is an isomorphism. For the map $\pi$, semisimplicity and degeneration is more subtle, and can be deduced from the vanishing of the Ceresa cycle. This will be explained in forthcoming work of the second author with Orsola Tommasi. However, one can be more explicit than this: for any $a \geq b \geq 0$, one has
$$ \dim H^3_c(\CV_2,V_{a,b}) = \dim W_0 H^3_c(\M_2,\mathbb V_{a,b}) = \begin{cases} \lfloor \tfrac{a-b}{6}\rfloor+1 & a \equiv b \equiv 1 \pmod 2\\ \lfloor \tfrac{a-b}{6}\rfloor & a \equiv b \equiv 0 \pmod 2\\ 0 & a \not\equiv b  \pmod 2\end{cases}, $$
and if $k\neq 3$ then $H^k_c(\CV_2,V_{a,b})=W_0H^k_c(\M_2,\mathbb V_{a,b})=0$. 
 See the discussion in \cite[Section 6.1]{gadish-hainaut}. The second equality can be deduced from the main theorem of \cite{localsystemsa2} and the branching rule of \cite[Proposition 3.4]{branching}.

Assuming \cref{conjecture 2}, it follows from \cref{multiplicity thm} that $$H^p_c(\CV_g,(R^q \rho_! \Q)^{\mathrm{ss}}) \cong W_0H^p_c(\M_g,(R^q\pi_!\Q)^{\mathrm{ss}})$$
for all $g\geq 2$, where $(-)^{\mathrm{ss}}$ denotes semisimplification. In promoting this to an isomorphism without the semisimplifications, i.e.\ $H^p_c(\CV_g,R^q \rho_! \Q) \cong W_0H^p_c(\M_g,R^q\pi_!\Q)$, one is led to comparing the Ext-groups of the various representations $V_\lambda$ of $\mathrm{Out}(F_g)$ with the Ext-groups of the local systems (or variations of Hodge structure) $\mathbb V_\lambda$ on the moduli space $\M_g$. It would be interesting if nontrivial extensions of $\mathrm{GL}(g)$-representations inside $R^q \rho_! \Q$ could be meaningfully ``matched'' with nontrivial extensions of $\mathrm{Sp}(2g)$-representations inside $R^q \pi_!\Q$. Such extensions are detected by elements of the Torelli subgroups of $\mathrm{Out}(F_g)$, resp.\ $\Mod_g$, acting nontrivially on $H^q_c(F(\vee_g S^1,n),\Q)$, resp.\ $H^q_c(F(\Sigma_g,n),\Q)$. The question of (non)triviality of the action of the Torelli group on configuration spaces of points of surfaces has been intensely studied in recent years \cite{moriyama,looijenga-torelli,bianchimillerwilson,bianchi-torelli,stavrou,bianchi-stavrou}; so has the analogous question for $\mathrm{Out}(F_g)$ and configurations of points on wedges of circles, and construction of nontrivial extensions of the representations $V_\lambda$ of $\mathrm{Out}(F_g)$ \cite{turchin-willwacher,vespa,powell-vespa,gadish,gadish-hainaut}.

\bibliographystyle{alpha}
\bibliography{database}

\end{document}